\documentclass[12pt]{amsart}
\usepackage{amsmath,amsthm,amscd,amssymb}
 \pdfoutput=1
\usepackage[margin=.75 in]{geometry}
\usepackage{subfig}

\usepackage{graphicx,placeins,wrapfig,sidecap} 
\usepackage[bookmarks=false]{hyperref}
\usepackage[labelfont=bf]{caption}
\allowdisplaybreaks
\usepackage{color}
\usepackage{enumitem}

\newcommand{\mb}{\mathbf}
\newcommand{\bs}{\boldsymbol}
\newcommand{\R}{\mathbb{R}}

\newcommand{\nn}{\nonumber}

\theoremstyle{plain}
\newtheorem{theorem}{Theorem}[section]
\newtheorem{theorem_star}{Theorem*}[section]
\newtheorem{lemma}[theorem]{Lemma}

\newtheorem{proposition}[theorem]{Proposition}
\newtheorem{remark}[theorem]{Remark}

\author{Aaron Zeff Palmer}
\thanks{azp@math.ucla.edu, University of California, Los Angeles, \today}

\title{The role of information in high dimensional stochastic optimal control}

\begin{document}

\begin{abstract}
	The stochastic optimal control of many agents is an important problem in various fields.  We investigate the problem of partial observations, where the state of each agent is not fully observed and the control must be decided based on noisy observations.  This results in a high-dimensional Markov decision process that is impractical to handle directly.  However, in the limit as the number of agents approaches infinity, a finite-dimensional mean-field optimal control problem emerges, which coincides with the problem of full information.

	Our main contribution is to investigate a central limit theorem for the Gaussian fluctuations of the mean-field optimal control.  Our findings show that partial observations play an essential role in the fluctuations, in contrast to the mean-field limit.  We establish a method that uses an approximate Kalman filter, which is straightforward to compute even when the number of states is large. This provides some theoretical evidence of the efficacy of Kalman filter methods that are commonly used across a range of practical applications. We demonstrate our results with two examples: an epidemic model with observations of positive tests and a simple two-state model that exhibits a phase transition at which point the fluctuations diverge.
\end{abstract}

\maketitle
\tableofcontents

\section{Introduction}

The role of information in stochastic control has been central to many applications. The Kalman filter was developed to predict a dynamically evolving state with noisy observations \cite{kalman1960new},\cite{kalman1961new} and has been employed in aeronautics, navigation, and economics, among other fields.  The Kalman filter provides part of the well-known solution to the linear-quadratic-Gaussian (LQG) control problem, which is the basis of this study.  

While individual behavior can be impossible to predict, systems comprised of many individuals tend to develop recognizable behaviors.  To give an economic perspective to the problem, we highlight two influential works in economics, amid a vast research field.  The role of information was recognized in \cite{von2016economic} as part of the critique that a socialist economy cannot adapt to individual needs as the information on the individual is lost, albeit the argument is made without mathematical formalism.    Later in \cite{greenwald1986imperfect}, \cite{greenwald1984informational}  imperfect information is shown to give rise to inefficient free markets using mathematical models.  

In this work, we study the optimal policies given available observations, such as by a government maximizing the well-being of the population or by a private enterprise maximizing its profit.  We present a solution that both closely approximates the optimal policy and may be computed using well-established and computationally efficient techniques.

The concept of a mean-field limit originates in physics and provides a simplified framework to study systems of many interacting particles. The central assumption of this framework is that the particles only interact through a macroscopic field, corresponding to their mean values. Connections between mean-field theories and partial observation stochastic control have been developed in the recent works of \cite{bandini2019randomized}, \cite{bensoussan2021mean}. The fluctuations of the mean-field limit capture how the system of particles deviates from its mean. For the stochastic optimal control of many agents with partial information, we find that the role of information is in the fluctuations of the mean-field limit. 
% We emphasize that the information is not found in the mean-field limit due to strong assumptions that cause individual agents to become uncorrelated in the limit. Without such assumptions, one would expect the information might also affect the mean-field limit.  A key takeaway from this work is that the role of information will be through fluctuations and correlations of agents.

We consider a mathematical framework of a finite-state Markov decision process with a parallel finite-state observation process.  We assume the problem involves a large number, $N$, of identical agents each with a fixed number of states.  Adding additional states allows for further specification of the individual agents.  In this framework, each individual produces observations at rates depending on their state.   We consider a fixed number of control parameters to be chosen at each time based on the available observations.  The cost is given by functions of the macroscopic state variables, which aggregate the individuals in each state. Some socio-economic applications of related finite-state mean field games have been given in \cite{gomes2014socio}.

Our first inquiry is of the limit as the number of agents, $N$, approaches infinity.   In this asymptotic limit, the (discrete) finite-state stochastic control problem approaches a (continuous) finite-dimensional nonlinear problem of (deterministic) optimal control. We call this problem the mean-field limit. The deterministic nature of the limiting control problem allows the state to be predicted, and the optimal control can be computed without using any observations.  A similar remark has been made for related problems in \cite{carmona2018probabilistic}, Remark 2.27, and in a context closely related to our own in \cite{cecchin2019convergence}.  Alternatively, mean-field control problems where the information of individual agents have been considered, for example in \cite{hafayed2015mean} where they develop the stochastic maximum principle approach and calculate the solution of a linear-quadratic-Gaussian example.  The problem of rigorous computational approaches to such problems beyond linear-quadratic-Guassian assumptions appears to be unsolved.

We next study the fluctuations of the mean-field limit.  Such fluctuations in Markovian population models were obtained in  \cite{kurtz1971limit},  \cite{barbour1974functional},  \cite{norman1974central}, \cite{shiga1985central}, \cite{dawson1991law}.  For the fully observed controlled version, analysis of the fluctuations has been carried out in \cite{zhang2015optimal}, \cite{cecchin2019convergence} for finite-state spaces. Furthermore, the approach has been extended to continuous-state spaces in \cite{delarue2019master} where the master equation is used to find a feedback control form.  Rather than focusing on the master equation, which involves additional complications with partial observations, we focus on a localized analysis around a mean-field trajectory.  

The localized asymptotic analysis we carry out involves solving a linear-quadratic-Gaussian control (LQG) problem using the problem data approximated quadratically about a mean-field solution.  This approximation follows an approach similar to \cite{benigno2012linear}, here it was used for fully observed optimal control.  The solution to the LQG problem satisfies the separation principle, where the Kalman filter is used to estimate the state, and the dynamic programming principle readily provides a linear feedback control.  Computation of these solutions involves solving decoupled forward and backward Ricatti-type systems of ordinary differential equations.  Given these solutions, we can then express an approximate feedback control policy, which we show is accurate to the first order ($N^{-1}$) correction of the cost.

We analyze two illustrative examples.  The first example is motivated by the Ising model of statistical physics and provides a glimpse of the behavior near a phase transition.  In this example, we see clearly that the fluctuations diverge at the critical point for the phase transition.

Our second example is to consider a simple epidemic model with observations corresponding to tests of individuals.  This example showcases the computational and practical possibilities of our approach.  While most of this paper is written for continuous time, the approach applies to discrete time, which is used in this example.

We briefly mention a few extensions and related problems that may be of interest to keep in mind.
\begin{itemize}
	\item While we consider `global' controls with a centralized planner, it includes the case of individual controls, where each agent selects a control $\alpha'\in\R^{m'}$ with the same access to 'global' information and full information of their state.  Since agents are identical and have the same information, the optimal policy will have the same control for any agents in the same state, and thus we can reduce to global controls by considering the control in a feedback form, $\alpha\in \R^{l\, m'}$ where $l$ denotes the number of states. The case when individual agents have partial observations of their state is similar to what appears in \cite{hafayed2015mean} and other works, and would also be interesting in our context.
	\item A game-theoretic problem where each player has the same information and is in a Nash equilibrium can be considered.  In the special case of a potential game, Nash equilibria can be found as optimal policies within our framework.  Otherwise, the Nash equilibrium constraint introduces additional complications.  While mathematically elegant, the assumption that each player behaves optimally is unlikely to hold in applications.  See for instance \cite{epstein2008ambiguity}, where the behavior of financial investors with imperfect information is modeled based on experiments.
	\item Market price is commonly considered in economic models by including a market-clearing constraint.  The convergence of a finite-state model to its mean-field limit has been considered in the recent work \cite{fujii2020finite}.

	\item Continuous-space problems present many interesting challenges.  The problems we consider may come from a discretization of continuous space problems by `coarse-graining' nearby particles into the same state and approximating the state transitions.  Understanding the continuum limits can be a highly difficult problem.  It has been discussed for mean-field games in \cite{delarue2019master}, although including partial observation has its own challenges.
\end{itemize}

\section{Many agent problem}

We consider $N$ identical agents with states $\sigma^{i}_t\in \{1,\ldots,l\}$ for $i\in \{1,\ldots,N\}$ and a global control $\alpha_t \in \R^m$.  We let $\Sigma_t^N \in \mathcal{\R}^l$ denote the empirical measure of states, so that $\Sigma_t^N = \frac{1}{N}\sum_{i=1}^N \delta_{\sigma^i_t}$.  We make the important simplifying `mean-field' assumption that all interactions are through this empirical measure.  

An agent in state $\sigma$ transitions to state $\gamma$ at rate $\beta(\sigma,\gamma,\Sigma_t^N,\alpha_t)$. If $\alpha_t$ is constant, then this makes the time until the next transition an exponential random variable with a rate of
$
	\sum_{i=1}^N \sum_{\gamma=1}^l \beta(\sigma_i,\gamma,\Sigma,\alpha).
$
 The transitions from $\Sigma_t^N \rightarrow \Sigma_t^N +\frac{1}{N}\big(\delta_\gamma - \delta_\sigma\big) $, accounting for all individuals in state $\sigma$, occur at rate $N\, \Sigma_t^{N,\sigma}\, \beta(\sigma,\gamma,\Sigma_t^N,\alpha_t^N)$, where $N\, \Sigma_t^{N,\sigma}$ is the number of individuals in stat $\sigma$.  

These dynamics can be expressed by the equations
\begin{align}\label{eqn:N_dynamics}
	 \Sigma_t^{N,\sigma} =&\ \Sigma_0^{N,\sigma} + \sum_{\gamma = 1:\gamma\not=\sigma}^l\Big(\frac{1}{N}Y^{\gamma\, \sigma}_t  - \frac{1}{N}Y^{\sigma\, \gamma}_t\Big),
\end{align}
where each $Y^{\sigma\, \gamma}$ is an independent Poisson process with non-homogeneous rate
$$
	\lambda_t^{\sigma\, \gamma} = N\, \Sigma_t^{N,\sigma}\, \beta(\sigma,\gamma,\Sigma_t^N,\alpha_t^N).
$$ 
There is always a unique (in law) solution $\Sigma_t^N$, see \cite{gihman2012controlled} for details on constructing such processes.  We will abbreviate the paths as
$$
	(\Sigma^N) = (\Sigma_t^N)_{t\in [0,T]}.
$$

We consider also cumulative measurements $\Upsilon_t^N \in \R^{\tilde{l}}$, where $\Upsilon_t^{N,\upsilon}\rightarrow \Upsilon_t^{N,\upsilon}+\frac{1}{N}$ at rate $N\, \Sigma_t^{N,\sigma}\, \tilde{\beta}(\sigma,\upsilon,\Sigma_t^N)$ and thus the dynamics of $\Upsilon^N_t$ can be expressed as
\begin{align}\label{eqn:Upsilon_dynamics}
	\Upsilon_t^{N,\upsilon} =&\ \frac{1}{N}\sum_{\sigma=1}^l\tilde{Y}^{\sigma\, \upsilon}_t,
\end{align}
where again each $\tilde{Y}_t^{\sigma\, \upsilon}$ is an independent non-homogeneous Poisson process of rate $N\, \Sigma_t^{N,\sigma}\, \tilde{\beta}(\sigma,\upsilon,\Sigma_t^N)$.  We let $(\mathcal{G}^N)=(\mathcal{G}^N_t)_{t\in [0,T]}$ be the natural filtration of $(\Upsilon^N)$, and $(\mathcal{F}^N)$ be the natural filtration of $(\Sigma^N,\Upsilon^N)$ so that $\mathcal{G}_t^N\subset \mathcal{F}_t^N$.  The filtration $(\mathcal{G}^N)$ represents the information available through the observations of $(\Upsilon^N)$.

  The finite time-horizon problem is to find $(\alpha^N) = (\alpha_t^N)_{t\in [0,T)}$ to minimize
\begin{align}\label{eqn:finite_N_cost} 
	J\big[(\Sigma^N),(\alpha^N)\big]=
		\mathbb{E}\Big[\int_0^T L\big(\Sigma_t^N,\alpha_t^N\big)dt + G\big(\Sigma_T^N\big)\Big] 
\end{align}
over control policies dependent on the history of measurements, $\alpha_t^N = \hat{\alpha}^N_t\big((\Upsilon_s^N)_{s\in [0,t]}\big)$ for a measurable sequence of functions ($\hat{\alpha}^N$).  Equivalently, $(\alpha^N)$ is progressively measurable with respect to the filtration $(\mathcal{G}^N)$.

\begin{remark}
	There are many possible generalizations of these controlled processes.  The simplest is to add the possibility of births and deaths, which would cause no problem for our analysis and we omit simply to reduce the burden of notation.  Inhomogeneity in time for $\beta$ and $L$ could also be easily handled.

	A very interesting case would be when $\beta$ has additional dependence on $N$.  This includes, for example, when each particle has a given position on a lattice, $x^{N,i}\in \mathbb{T}^2$, and interactions take place with nearest neighbors on the lattice.  This may result in much more complicated phenomena and will be partly explored by the author in related work. 
\end{remark}

We will consider the mean behavior and the fluctuations of the mean.  To anticipate these results it is useful to consider the Doob-Meyer decomposition of $\Sigma_t^N$, which takes the form
\begin{align}\label{eqn:Doob}
	\Sigma_t^N = \Sigma_0^N + \int_0^tb(\Sigma_s^N,\alpha_s^N)ds + M_t^N,
\end{align}
where the expected drift is given by
\begin{align}\label{eqn:drift}
	b^\sigma(\Sigma,\alpha) =  \sum_{\gamma=1:\gamma\not=\sigma}^l \Big(\Sigma^\gamma\, \beta(\gamma,\sigma,\Sigma,\alpha)-\Sigma^\sigma\, \beta(\sigma,\gamma,\Sigma,\alpha)\Big),
\end{align}
and the covariation of the martingale term, $(M^N)$, sums the squared jumps so that for $0\leq r\leq t$,
\begin{align}\label{eqn:covariation}
	E\Big[ M^{N,\sigma}_t\, M^{N,\gamma}_t\Big| \mathcal{F}_r^N\Big] =M^{N, \sigma}_r\, M^{N, \gamma}_r+\frac{1}{N}\mathbb{E}\Big[\int_r^t\epsilon(\Sigma_s^N,\alpha_s^N)\epsilon(\Sigma_s^N,\alpha_s^N)^\top ds\Big|\mathcal{F}_r^N\Big],
\end{align}
where $\epsilon$ is an $l$ by $l(l-1)$ matrix with entries, supposing the column $\nu=(\sigma,\gamma)$ corresponds to the transition from $\sigma$ to $\gamma\not=\sigma$, 
\begin{align*}
	[\epsilon(\Sigma,\alpha)]^{\delta\nu} = \begin{cases}\sqrt{\Sigma^\gamma \beta(\gamma,\sigma,\Sigma,\alpha)} & \delta=\sigma\\
	-\sqrt{\Sigma^\gamma \beta(\gamma,\sigma,\Sigma,\alpha)} & \delta=\gamma\\
	0 & {\rm otherwise},
	\end{cases}
\end{align*}
or for the matrix product
\begin{align*}
	[\epsilon(\Sigma,\alpha)\epsilon(\Sigma,\alpha)^\top]^{\sigma\gamma} = \begin{cases} \sum_{\gamma \not=\sigma}\Big( \Sigma^\gamma\, \beta(\gamma,\sigma,\Sigma,\alpha) + \Sigma^\sigma\, \beta(\sigma,\gamma,\Sigma,\alpha)\Big) & \sigma=\gamma\\
	-\Sigma^\gamma\, \beta(\gamma,\sigma,\Sigma,\alpha) - \Sigma^\sigma\, \beta(\sigma,\gamma,\Sigma,\alpha) & \sigma\not=\gamma.
	\end{cases}
\end{align*}
Basic results on uncontrolled models can be found in \cite{kurtz1971limit} and many subsequent works. Four different representations of the martingale term for similar Markovian queue models are given in \cite{pang2007martingale}, along with the proof of a central limit governing the fluctuations about a high-intensity limit.

Illustrative examples of the problem we study are given and worked out in Section \ref{sec:examples}.

\section{Mean-field approximation}
The mean-field approximation ignores the stochastic martingale term of the Doob-Meyer decomposition (\ref{eqn:Doob})  and recovers the limiting behavior as $N\rightarrow \infty$. In the mean-field problem, the empirical state measure $(\Sigma^N)$ becomes a deterministic trajectory $(S) = (S_t)_{t\in [0,T]}$, the control $(\alpha^N)$ becomes $(A) = (A_t)_{t\in [0,T]}$, and the transition rates are replaced by a drift $b(S,A)$ from (\ref{eqn:drift}).
The resulting dynamics are necessarily nonlinear in $(S, A)$.

The mean-field problem considers (weak) solutions to the dynamics
\begin{align}\label{eqn:weak_mf_dynamics}
	S_t = S_0+ \int_0^tb(S_s,A_s)ds
\end{align}
with cost
\begin{align}\label{eqn:mf_cost}
	J^{MF}\big[(S),(A)\big] =
		\int_0^TL(S_t,A_t)\, dt + G(S_T).
\end{align}

The following theorem is a basic result for the mean-field theory of stochastic optimization problems.  We note that the state $\Sigma$ naturally takes values in the probability simplex, which we denote by $\Delta^l\subset \R^l$.  We make the following assumptions on the problem data:
\begin{enumerate}[label={\normalfont \bf A\arabic* }]
	\item \label{itm:continuity} We assume that $\beta$, $L$, are jointly continuous in $\Sigma$ and $\alpha$. We assume that $L$ is uniformly bounded below and $\beta$ is uniformly bounded.
	\item \label{itm:coercivity} We assume that $L$ is coercive in the sense that $\{\alpha;\ L(\Sigma,\alpha)\leq M\}$ is compact for each $\Sigma\in \Delta^l$ and $M\in \R$. 
	\item \label{itm:convexity} We assume the standard convexity condition that for each $\Sigma\in \Delta^l$, the set
	$$
		\Big\{\big(m,z\big)\in \R^{l+1};\ \exists\, \alpha\ {\rm s.t.}\ b(\Sigma,\alpha)=m\ \hbox{and} \ L(\Sigma,\alpha)\leq z\Big\}
	$$
	is convex. In the case that $\alpha\mapsto b(\Sigma,\alpha)$ is linear (as in our examples) this assumption reduces to that $\alpha\mapsto L(\Sigma,\alpha)$ is convex.
\end{enumerate}

We use the notation $X^N\rightarrow^d X$ to denote convergence in distribution for the random variable, i.e., weak convergence. In Theorem \ref{thm:mean_field}, weak convergence is equivalent to convergence in probability because the limits are deterministic.

There are several approaches to analyzing the mean-field limit of the $N$-player problem, from which we will borrow some results.
\begin{itemize}
	\item The $\Gamma$-convergence approach. Our main arguments are inspired by this approach, but we do not attempt to show the full $\Gamma$-convergence. This approach is well-suited as it does not require convergence of the control policy, which we do not expect in all cases.

	\item The approach of Young-measures.  We borrow the concept from Young-measures of a compactification of the control policies. Assumption \ref{itm:convexity} enables us to pass from a Young-measure randomized control back to a control policy using a sort of barycenter.

	\item The approach by Hamilton-Jacobi equation.  Since we are interested in applications for high-dimensional problems the Hamilton-Jacobi equation does not provide a very practical approach and also cannot easily incorporate partial information.  However, when we address the fluctuations in Section \ref{sec:fluctuations} we will require estimates that closely parallel results from the Hamilton-Jacobi equations.
\end{itemize}

The assumption that the mean-field problem has a unique minimizer is a simplifying assumption to avoid obfuscation from having randomized limits.
\begin{theorem}\label{thm:mean_field}
	We assume \ref{itm:continuity}, \ref{itm:coercivity}, \ref{itm:convexity} and suppose that $\Sigma_0^N\rightarrow^d S_0\in \Delta^l$.

	If the mean-field problem has unique minimizer, $(S^*,A^*)$, and $(\Sigma^N,\alpha^N)$ are optimizers of the $N$-player problem {\normalfont(\ref{eqn:finite_N_cost})}, then $(\Sigma)^N\rightarrow^d (S^*)$ as $N\rightarrow \infty$ and $\lim_{N\rightarrow \infty} J\big[(\Sigma^N),(\alpha^N)\big] = J^{MF}\big[(S^*),(A^*)\big]$.

	Furthermore, $\tilde{\alpha}^N_t = A^*_t$ is an approximate optimal control for the $N$-player problem in the sense that 
	$\lim_{N\rightarrow \infty} J\big[(\tilde{\Sigma}^N),(\tilde{\alpha}^N)\big] = J^{MF}\big[(S^*),(A^*)\big]$.
\end{theorem}
\begin{proof}

We let $\Omega$ denote the Skorokhod space of c\`{a}dl\`{a}g paths from $[0,T]$ into $\Delta^l\times \R^{\tilde{l}}$.  We first obtain tightness for the distributions of $(\Sigma^N,\Upsilon^N)$, which follows from \ref{itm:continuity} by considering that $(\Sigma^{N,\sigma})$ and $(\Upsilon^{N,\upsilon})$ decompose as the difference of increasing processes in (\ref{eqn:N_dynamics}) and (\ref{eqn:Upsilon_dynamics}).  We next compactify the control $(\alpha^N)$ in the space of measures, $\mu^N_t \in \mathcal{P}(\Omega \times \R^m)$, such that for $f\in C_b(\Omega\times \R^m)$, 
\begin{align}\label{eqn:control_compactification}
	\int_{\Omega\times  \R^m}f(\omega,a)\mu_t^N(d\omega, da) = \mathbb{E}\big[f(\cdot,\alpha_t^N)\big].
\end{align}
Assumption \ref{itm:coercivity} and boundedness of the cost implies tightness for these measures $(\mu^N)$.  We let $\hat{\mu}_t^N:\Omega \rightarrow \mathcal{P}(\R^m)$ denote the disintegration of the measure with respect to the distribution of $(\Sigma^N,\Upsilon^N)$.
For every subsequence $\{N_i\}_{i=1}^\infty$ we find a subsubsequence $\{N_{i_j}\}_{j=1}^\infty$ such that $(\Sigma^{N_{i_j}})\rightarrow^d (\Sigma)$ and $(\hat{\mu}^{N_{i_j}})\rightarrow^d \hat{\mu}$.

 Using \ref{itm:continuity},  we have that
	$$
		\Sigma_t =\Sigma_0 + \int_0^t\int_{\R^m} b(\Sigma_s,a)\hat{\mu}_s(da)\, ds
	$$
	holds almost surely,
	and 
	$$
		\mathbb{E}\Big[ \int_0^T\int_{\R^m}L(\Sigma_t,a)\hat{\mu}_t(da)\, dt\Big]\leq \liminf_{j\rightarrow \infty} \mathbb{E}\Big[\int_0^T L(\Sigma_t^{N_{i_j}},\alpha_t^{N_{i_j}})dt\Big].
	$$
	By \ref{itm:convexity} and the Kuratowski-Ryll-Nardzewski measurable selection theorem, we find a random measurable control policy $(\alpha):\Omega\times[0,T]\rightarrow \R^m$  such that, for almost every $t$, 
	$$
		\int_{\R^m} b(\Sigma_t,a)\hat{\mu}_t(da) = b(\Sigma_t,\alpha_t)
	$$
	and
	$$
		\int_{\R^m}L(\Sigma_t,a)\hat{\mu}_t(da)\, dt \geq  L(\Sigma_t,\alpha_t)
	$$
	hold almost surely. (In the case when $\alpha\mapsto b(\Sigma,\alpha)$ is linear and $\alpha\mapsto L(\Sigma,\alpha)$ is convex, this is given simply by the barycenter, $\alpha_t=\int_{\R^m} a\, \hat{\mu}_t(da)$.)

	We obtain the lower bound inequality for the sequence of minimizers 
	$$
	\mathbb{E}\Big[J^{MF}\big[(\Sigma),(\alpha)\big]\Big] = J\big[(\Sigma),(\alpha)\big]\leq \liminf_{j\rightarrow \infty} J\big[(\Sigma^{N_{i_j}}),(\alpha^{N_{i_j}})\big],
	$$
	and $(\Sigma)$, $(\alpha)$ is a (randomized) weak solution to the mean-field dynamics (\ref{eqn:weak_mf_dynamics}).

We now consider the optimizer of the mean-field problem, $(S^*)$ and $(A^*)$.  Since, $(A^*)$ is deterministic, we can directly use $\tilde{\alpha}^N = A^*$ as a policy for the $N$ player problem with cost $J[(\tilde{\Sigma}^N),(\tilde{\alpha}^N)]$.  Note that $\tilde{\alpha}^N$ makes no use of information.  We now observe, using tightness and continuity as above, that $(\tilde{\Sigma}^N)\rightarrow^d (S^*)$ as $N\rightarrow \infty$ and 
$$
	\lim_{N\rightarrow \infty} J\big[(\tilde{\Sigma}^N),(\tilde{\alpha}^N)\big] = J^{MF}\big[({S}^*),({A}^*)\big],
$$
which proves the final statement of the theorem.

We now also infer the optimality of the limit of optimizers, $(\Sigma)$, because
$$
\mathbb{E}\Big[J^{MF}\big[(\Sigma),(\alpha)\big]\Big]\leq \liminf_{j\rightarrow \infty} J^{N_{i_j}}\big[(\Sigma^{N_{i_j}}),(\alpha^{N_{i_j}})\big]\leq \lim_{j\rightarrow \infty} J^{N_{i_j}}\big[(\tilde{\Sigma}^{N_{i_j}}),(\tilde{\alpha}^{N_{i_j}})\big] = J^{MF}\big[(S^*),(A^*)\big].
$$  
The inequality $J^{MF}\big[(\Sigma),(\alpha)\big]\geq J^{MF}\big[(S^*),(A^*)\big]$ implies that equality holds almost everywhere and thus $(\Sigma)=(S^*)$ almost surely.

Since every subsequence of $(\Sigma^N)$ has a subsubsequence converging, $(\Sigma^{N_{i_j}})\rightarrow^d (S^*)$ as $j\rightarrow \infty$, we conclude that $(\Sigma^N)\rightarrow^d (S^*)$ as $N\rightarrow \infty$ and $\lim_{N\rightarrow \infty} J\big[(\Sigma^N),(\alpha^N)\big] = J^{MF}\big[(S^*),(A^*)\big]$.
\end{proof}

\subsection{Optimality Criterion}
Having obtained a minimizer of the mean-field problem (\ref{eqn:mf_cost}) we consider the first-order optimality criteria. 
We call $P\in \mathbb{R}^l$ the co-state and define the generalized Hamiltonian to be
$$
	\mathcal{H}(S,A,P) = P\cdot b(S,A) - L(S,A).
$$

Given an optimal trajectory $(S^*)$, $(A^*)$, we let $(P^*)$ solve the co-state equation, with $S^*_0=S_0$, $P_T^* = -\nabla G(S_T^*)$, so that $(S^*)$ and $(P^*)$ solve
\begin{align}\label{eqn:Hamiltonian_system}
\frac{dS_t^*}{dt} =&\ D_P\mathcal{H}(S^*_t, A^*_t, P_t^*)=b(S^*_t, A^*_t)\\
-\frac{dP_t^*}{dt} =&\ D_S\mathcal{H}(S^*_t,A^*_t,P^*_t).\nn
\end{align}
The operators $D_P$ and $D_S$ are the partial derivatives with respect to the co-state and state, respectively.
The Pontryagin maximum principle states that at points of continuity of $(S^*)$ and $(P^*)$,
\begin{align}\label{eqn:max_principle}
A_t^* \in&\ {\rm argmax}\big\{ \mathcal{H}(S^*_t,\cdot,P^*_t)\big\}.
\end{align}
Assuming some smoothness of the trajectories, $P^*$ provides the first variation of the cost.  A variation of the position at time $t$, $S_t\rightarrow S_t+\delta\, \xi_t$ yields a change in the optimal cost of $-\delta\, \xi_t\cdot P_t^* +o(\delta)$.

While (\ref{eqn:Hamiltonian_system}) and (\ref{eqn:max_principle}) provide necessary conditions, if a unique solution exists with the minimal cost then Theorem \ref{thm:mean_field} implies this is the mean-field limit of the $N$-agent problem.

The mean observations are easily recovered by
\begin{align}\label{eqn:mean_field_observations}
	\frac{dU^*_t}{dt} = \tilde{b}(S_t^*),
\end{align}
where
$$
	\tilde{b}(\Sigma)^\upsilon = \sum_{\sigma=1}^l\beta(\sigma,\upsilon,\Sigma)\, \Sigma^\sigma.
$$

\subsection{Propagation of Chaos}
	It is often interesting to see what Theorem \ref{thm:mean_field} says about individual particles.  A propagation of chaos result states that for each agent, its limiting distribution corresponds to an independent Markov process with the mean-field empirical measure fixed.

	More generally, choose $k$ agents with states $\{\sigma^{1,N}_t,\sigma^{2,N}_t,\ldots,\sigma^{k,N}\}$, then the joint distribution of these particles converges to the distribution of $k$ independent agents, where each agent satisfies the Markovian transitions determined by
	$$
		\beta(\sigma^i,\gamma,S_t,A_t).
	$$

\section{Fluctuations of the Mean Field Limit}\label{sec:fluctuations}

We aim to refine the mean-field approximation to capture the Gaussian fluctuations of the limit.  We assume $(S^*)$ and $(A^*)$ are the unique minimizers of the mean-field problem and $(P^*)$ is the co-state that solves (\ref{eqn:Hamiltonian_system}).  We also let $(U^*)$ be the mean observations that solve (\ref{eqn:mean_field_observations}).
We will show that the state, observations, and control can be expanded as
\begin{align*}
 \Sigma_t^N =&\ S_t^*+N^{-1/2}\, \mathfrak{s}_t^* + o(N^{-1/2}),\\
\Upsilon_t^N =&\ U_t^* + N^{-1/2}\, \mathfrak{u}_t^* + o(N^{-1/2}),\\
\alpha_t^N =&\ A_t^* + N^{-1/2}\, \mathfrak{a}_t^* + o(N^{-1/2}),
\end{align*}
where $\mathfrak{s}_t^*$, $\mathfrak{u}_t^*$ and $\mathfrak{a}_t^*$ have mean zero, and minimize a linear-quadratic-Gaussian problem.

\subsection{Linear-Quadratic-Gaussian Approximation}
The linear approximation of the dynamics for the state and observations are
\begin{align}\label{eqn:linearized_dynamics}
	d\mathfrak{s}_t =&\ Db(S_t^*,A_t^*)\big[(\mathfrak{s}_t,\mathfrak{a}_t)\big]dt +\sqrt{2}\, \epsilon(S_t^*,A_t^*)\, d{W}_t,\\
	d\mathfrak{u}_t =&\ D\tilde{b}(S_t^*)\big[\mathfrak{s}_t\big]dt +\sqrt{2}\, \tilde{\epsilon}(S_t^*)\, d\tilde{W}_t.\nn
\end{align}
We have used $D$ to represent the gradient operator, such that $Db(S,A)$ is a linear transformation mapping $\R^l\times \R^m\rightarrow \R^l$ and $D\tilde{b}(S):\R^l\rightarrow \R^{\tilde{l}}$.
The diffusion coefficients $\epsilon$ and $\tilde{\epsilon}$ are an $l$ by $l(l-1)$ matrix and a $\tilde{l}$ by $l(l-1)$ matrix that captures the quadratic variation of the process see (\ref{eqn:covariation}) and $(W)$ and $(\tilde{W})$ are $l(l-1)$ independent Brownian motions.    Specifically, when $\gamma\not=\sigma$
$$
	\epsilon^{\sigma(\sigma,\gamma)}(S,A) = -\sqrt{\beta(\sigma,\gamma,S,A)\, S^\sigma},
$$
and
$$
	\epsilon^{\sigma(\gamma,\sigma)}(S,A) = \sqrt{\beta(\gamma,\sigma,S,A)\, S^\gamma}.
$$
Similarly, 
$$
	\tilde{\epsilon}^{\upsilon(\upsilon,\sigma)}(S,A) = \sqrt{\tilde{\beta}(\sigma,\upsilon,S,A)\, S^\sigma}.
$$

We assume that $\Sigma^N_0 = S_0^* + N^{-1/2}\, \mathfrak{s}_0 + o(N^{-1/2})$ where $\mathfrak{s}_0$ is Gaussian with mean zero and $\mathbb{E}[\Sigma^N_0] = N^{-1}\, \zeta_0 + o(N^{-1})$. The cost will have the asymptotic expansion
$$
	J\big[(\Sigma^N), (\alpha^N)\big]= J^{MF}\big[(S^*),(A^*)\big] -N^{-1}\, P_0^*\cdot \zeta_0  + N^{-1}LQG\big[(\mathfrak{s}),(\mathfrak{a})\big] + o(N^{-1}),
$$
 where
the correction to the cost from the fluctuations is given by (dependence on $(S^*)$, $(A^*)$, and $(P^*)$ is implicit)
\begin{align}\label{eqn:LQG_cost}
	&\ LQG\big[(\mathfrak{s}),(\mathfrak{a})\big]\\
	 =&\  \mathbb{E}\Big[\int_0^T\Big(\frac{1}{2}D^2L(S_t^*,A_t^*) \big[(\mathfrak{s}_t,\mathfrak{a}_t),(\mathfrak{s}_t,\mathfrak{a}_t)\big] - \frac{1}{2}P_t^*\cdot D^2b(S_t^*,A_t^*) \big[(\mathfrak{s}_t,\mathfrak{a}_t),(\mathfrak{s}_t,\mathfrak{a}_t)\big]\Big) dt\nn\\
	&\ + \frac{1}{2}D^2G(S_T^*)\big[\mathfrak{s}_T,\mathfrak{s}_T\big]\Big]. \nn
\end{align}  
The control $(\mathfrak{a})$ will now be restricted to depend only on the linearized observations, $(\mathfrak{u})$.
The problem to minimize $LQG\big[(\mathfrak{s}),(\mathfrak{a})\big]$ with such a partial information constraint is well known \cite{kalman1961new}.

To express the solution, we define the following
\begin{align}\label{eqn:partial_info_relations}
	R_t =&\  \frac{1}{2}D_{AA}L(S^*_t,A^*_t)-\frac{1}{2}P_t^*\cdot D_{AA}b(S^*_t,A^*_t) \\
	B_t =&\ D_A b(S^*_t,A^*_t)\nn\\
	E_t =&\ D_S b(S^*_t,A^*_t)\nn\\
	Q_t =&\ \frac{1}{2}D_{SS} L(S^*_t, A^*_t)- \frac{1}{2}P^*_t\cdot D_{SS}b(S^*_t,A^*_t)\nn\\
	V_t =&\ \frac{1}{2}\Big(D_{SA} L(S^*_t,A^*_t) - P^*_t\cdot D_{SA}b(S^*_t,A^*_t)\Big)\nn\\
	\tilde{E}_t =&\ D_S \tilde{b}(S^*_t)\nn\\
	\Theta_t =&\ \epsilon(S^*_t,A^*_t)\epsilon(S^*_t,A_t)^\top\nn\\
	\tilde{\Theta}_t =&\ \tilde{\epsilon}(S^*_t)\epsilon(S^*_t)^\top\nn\\
	F =&\ \frac{1}{2}D^2G(S_T^*).\nn
\end{align}

Now we can express the linear quadratic cost in a standard form:
\begin{align*}
	LQG\big[(\mathfrak{s}),(\mathfrak{a})\big] = \mathbb{E}\Big[\int_0^T\Big(\mathfrak{s}_t\cdot Q_t\, \mathfrak{s}_t + 2\, \mathfrak{s}_t\cdot V_t\, \mathfrak{a}_t + \mathfrak{a}_t\cdot R_t\, \mathfrak{a}_t\Big)dt + \mathfrak{s}_T\cdot F\, \mathfrak{s}_T\Big]
\end{align*}
subject to
\begin{align}\label{eqn:linearized_dynamics}
	d\mathfrak{s}_t = \big(E_t\, \mathfrak{s}_t + B_t\, \mathfrak{a}_t\big)dt + d\theta_t,
\end{align}
where $\theta_t$ is a continuous martingale process with infinitesimal covariance (i.e., quadratic covariation) given by $\Theta(S^*_t,A^*_t)$.

Furthermore, we impose the information constraint that $\mathfrak{a}_t = \hat{\mathfrak{a}}_t((\mathfrak{u}_s)_{s\in [0,t]})$ , where the linearized observations are given by
\begin{align}\label{eqn:linearized_observation_dynamics}
	d\mathfrak{u}_t = \tilde{E}_t\, \mathfrak{s}_t\, dt + d\tilde{\theta}_t,
\end{align}
with $\mathfrak{u}_0=0$ and $\tilde{\theta}_t$ is a continuous martingale process with infinitesimal covariance $\tilde{\Theta}(S^*_t,A^*_t)$.

The first step to solve the LQG problem is to compute the estimator $\hat{\mathfrak{s}}_t = \mathbb{E}\big[\mathfrak{s}_t|\, (\mathfrak{u}_s)_{s\in [0,t]}\big]$.  The separation theorem states that computing the estimator decouples from determining the control problem, and the optimal control is given as the linear feedback optimizer of the full information LQ problem substituting in $\hat{\mathfrak{s}}$.

We compute the covariance matrix $\Pi_t= {\rm cov}(\mathfrak{s}_t-\hat{\mathfrak{s}}_t)$  by solving the forward matrix Ricatti differential equation
\begin{align}\label{eqn:Kalman_covariance}
	\frac{d\Pi_t}{dt} = E_t^\top\, \Pi_t + \Pi_t\, E_t - \Pi_t\, \tilde{E}_t^\top \tilde{\Theta}^{-1}_t \tilde{E}_t\Pi_t + \Theta_t
\end{align}
with $\Pi_0 ={\rm cov}(\mathfrak{s}_0)$.  Importantly, while the covariance of $\mathfrak{s}$ may depend on the control, the covariance of the estimation error difference does not. Equation (\ref{eqn:Kalman_covariance}) always has a solution on $[0,T]$ because $\Theta_t\geq 0$ so the semidefinite inequality, $\Pi_t\geq 0,$ is maintained.

We get the estimator $\hat{\mathfrak{s}}_t$ given control $(\mathfrak{a})$ by solving 
$$
	d\hat{\mathfrak{s}}_t = E_t\, \hat{\mathfrak{s}}_t\, dt + B_t\, \mathfrak{a}_t\, dt + \Pi_t\, \tilde{E}_t^\top \tilde{\Theta}^{-1}_t\, (d\mathfrak{u}_t-\tilde{E}_t\, \hat{\mathfrak{s}}_t\, dt),
$$
with $\hat{\mathfrak{s}}_0=0$.  Supposing that $\mathfrak{a}_t = \hat{\mathfrak{a}}_t((\mathfrak{u}_s)_{s\in [0,t]})$, we have computed $\hat{\mathfrak{s}}_t = \mathbb{E}\big[\mathfrak{s}_t|\,(\mathfrak{u}_s)_{s\in [0,t]}\big]$.

For the optimal control, we solve the backward Ricatti equation
\begin{align}\label{eqn:Ricatti}
	-\frac{dZ_t}{dt} = E_t^\top\, Z_t + Z_t\, E_t - (Z_t\, B_t + V_t)\, R^{-1}_t\, (Z_t\, B_t + V_t)^\top + Q_t
\end{align}
with $Z_T = F=\frac{1}{2}D^2_SG(S^*_T)$.  The optimal control is then given by 
$$
	\mathfrak{a}_t^* = -R^{-1}_t\,  (Z_t\, B_t + V_t)^\top \hat{\mathfrak{s}}_t^*.
$$
% A useful calculation is given by It\^{o}'s Lemma and (\ref{eqn:Ricatti})
% \begin{align*}
% 	d\, (\mathfrak{s}_t\cdot Z_t\, \mathfrak{s}_t)	=&\ \mathfrak{s}_t\cdot Z_t'\, \mathfrak{s}_t\, dt + d\mathfrak{s}_t\cdot Z_t\, \mathfrak{s}_t + \mathfrak{s}_t\cdot Z_t\, d\mathfrak{s}_t +2\, \Theta_t \cdot Z_t\, dt\\
% 	=&\ \Big(\mathfrak{s}_t\cdot (Z_t\, B_t + V_t)\, R^{-1}_t\, (Z_t\, B_t + V_t)^\top\, \mathfrak{s}_t- \mathfrak{s}_t\cdot Q_t\mathfrak{s}_t+B_t\, \mathfrak{a}_t\cdot Z_t\, \mathfrak{s}_t+\mathfrak{s}_t\cdot Z_t\, B_t\, \mathfrak{a}_t\Big)dt\\
% 	&\ 2\, \Theta_t\cdot Z_t\, dt+ d\theta_t\cdot Z_t\, \mathfrak{s}_t+\mathfrak{s}_t\cdot Z_t\, d\theta_t.
% \end{align*}

% The optimal cost-to-go is then given by
% \begin{align}\label{eqn:LQG_cost_to_go}
%  &\ \mathbb{E}\Big[\int_t^T\Big(\mathfrak{s}_r^*\cdot Q_r\, \mathfrak{s}_r^* + 2\, \mathfrak{s}_r^*\cdot V_r\, \mathfrak{a}_r^* + \mathfrak{a}_r^*\cdot R_r\, \mathfrak{a}_r^*\Big)dt + \mathfrak{s}_T^*\cdot F\, \mathfrak{s}_T^*\Big| (\mathfrak{u}_s)_{s\in [0,t]}\Big] \nn \\
%   =&\ \mathbb{E}\Big[\int_t^T\Big(-d\, (\mathfrak{s}^*_r\cdot Z_r\, \mathfrak{s}^*_r)+(\mathfrak{s}^*_r-\hat{\mathfrak{s}}_r)\cdot(Z_r\, B_r + V_r)\, R^{-1}_r\, (Z_r\, B_r + V_r)^\top\, (\mathfrak{s}_r^*-\hat{\mathfrak{s}}_r)\nn\\
%   &\ +2\,\Theta_r\cdot Z_r\Big)dr+ \mathfrak{s}_T^*\cdot Z_T\, \mathfrak{s}_T^*\Big| (\mathfrak{u}_s)_{s\in [0,t]}\Big] \nn \\
% =&\ \hat{\mathfrak{s}}_t\cdot Z_t\, \hat{\mathfrak{s}}_t + \Pi_t\cdot Z_t + \int_t^T\Big(\Pi_r\cdot (Z_r\, B_r + V_r)\, R^{-1}_r\, (Z_r\, B_r + V_r)^\top + 2\, \Theta_r\cdot Z_r\Big)dr.
% \end{align}

Equation (\ref{eqn:Ricatti}) may not have a solution on $[0,T]$, depending on the data.  We will see later in Section \ref{sec:examples_Ising} how this might break down.

\subsection{Central Limit Theorem}
We will now address the convergence to the solution of the LQG approximation.  We make the additional assumption of regularity for the mean-field solution and problem data:
\begin{enumerate}[label={\normalfont \bf A\arabic* }]
	\setcounter{enumi}{3}

	\item \label{itm:smooth_data}  We suppose that $(S_t^*,A_t^*,P_t^*)$ is smooth and $\beta$ and $L$ are smooth near in a neighborhood of $(S_t^*,A_t^*,P_t^*)$. 

	% \item \label{itm:convex_in_A}  We suppose that $R_t>0$ is positive definite and that there are $R_t^N$ converging uniformly to $R_t$ such that
	% 	$\alpha \mapsto L(S_t^*, A_t^* + \alpha) - P^*_t\cdot b(S^*,A_t^*+\alpha) -  R_t^N[\alpha,\alpha]$ is convex. 
\end{enumerate}

We begin with a lemma that relates the first-order asymptotic expansion of the cost to the linear quadratic cost, using the optimality criteria and Taylor expansions of the data.
\begin{lemma}\label{lem:cost_asymptotics}
	We assume \ref{itm:continuity} and \ref{itm:smooth_data}.  Suppose that $(\Sigma^N)$ and $(\alpha^N)$ satisfy {\normalfont(\ref{eqn:N_dynamics})}, and $(S^*)$, $(P^*)$, and $(A^*)$ satisfy {\normalfont(\ref{eqn:Hamiltonian_system})} and {\normalfont(\ref{eqn:max_principle})}.  Then we have
	\begin{align*}
		 &\ J\big[(\Sigma^N),({\alpha}^N)\big] - J^{MF}\big[(S^*),(A^*)\big]\\
		  =&\ LQG\big[(\Sigma^N-S^*),(\alpha^N-A^*)\big] - P_0^* \cdot \Big(\mathbb{E}\big[\Sigma^N_0\big]-S^*_0\Big)\\
		  &\ +   o\big(\sup_{t\in [0,T]}\mathbb{E}\big[|\Sigma^N_t-S^*_t|^2\big]+\mathbb{E}\big[\int_0^T|\alpha^N_t-A^*_t|^2dt\big]\big).
	\end{align*}
	% Equality holds (without \ref{itm:convex_in_A})  if we replace the integral in the error term with $\sup_{t\in [0,T]}\mathbb{E}\big[|\alpha^N_t-A^*_t|^2\big]$.
\end{lemma}
\begin{proof}
	The proof is a direct calculation using the Taylor expansion of the cost, the optimality criteria (\ref{eqn:Hamiltonian_system}) and (\ref{eqn:max_principle}), and an expansion of the drift in the dynamics of (\ref{eqn:N_dynamics}).
	\begin{align*}
	&\ J\big[(\Sigma^N),({\alpha}^N)\big] - J^{MF}\big[(S^*),(A^*)\big]\\
	 =&\ 	 \mathbb{E}\Big[\int_0^T\Big(L(\Sigma^N_t,\alpha^N_t) - L(S^*_t,A^*_t)\Big)dt + G(\Sigma^N_T) - G(S^*_T)\Big]\\\
	\geq &\ \mathbb{E}\Big[\int_0^T\Big(DL(S_t^*,A_t^*)\big[(\Sigma^N_t-S^*_t), (\alpha^N_t - A^*_t)\big]\Big) dt+ DG(S_T^*)\big[\Sigma^N_T-S^*_T\big]\Big] \ \ \ \big(I_1\big)\\
	&\ + \mathbb{E}\Big[\int_0^T\Big(\frac{1}{2}D^2L(S_t^*,A_t^*) \big[(\Sigma^N_t-S^*_t, \alpha^N_t - A^*_t),(\Sigma^N_t-S^*_t, \alpha^N_t - A^*_t)\big]\Big) dt\\
	&\ \ \ \ \ + \frac{1}{2}D^2G(S_T^*)\big[\Sigma^N_T-S^*_T,\Sigma^N_T-S^*_T\big]\Big]\ \ \ \ \big(I_2)\\
	&\ +   o\big(\sup_{t\in [0,T]}\mathbb{E}\big[|\Sigma^N_t-S^*_t|^2\big]+\mathbb{E}\big[\int_0^T|\alpha^N_t-A^*_t|^2dt\big]\big).
\end{align*}
The terms of $(I_2)$ are part of our conclusion, and to handle the  term $(I_1)$ we have
 \begin{align*}
	(I_1) =&\  \mathbb{E}\Big[\int_0^T\Big(D_SL(S_t^*,A_t^*)\big[\Sigma^N_t-S^*_t\big] + D_AL(S_t^*,A_t^*)\big[\alpha^N_t - A^*_t\big]\Big) dt-P_T^*\cdot (\Sigma^N_T-S^*_T)\Big]\\
	=&\  \mathbb{E}\Big[\int_0^T\Big(D_SL(S_t^*,A_t^*)\big[\Sigma^N_t-S^*_t\big] + D_AL(S_t^*,A_t^*)\big[\alpha^N_t - A^*_t\big]+D_SH(S^*_t,A^*_t)\cdot  (\Sigma^N_t-S^*_t) \Big) dt\\
	&\  - \int_0^TP_t^*\cdot  (d\Sigma^N_t-dS^*_t) - P_0\cdot (\Sigma^N_0-S_0)\Big].
\end{align*}
	We now use (\ref{eqn:max_principle}) to equate 
	$$D_AL(S_t^*,A_t^*)\big[\alpha^N_t - A^*_t\big] = P_t^* \cdot D_Ab(S_t^*,A_t^*)\big[\alpha^N_t - A^*_t\big]
	$$
	and we use the Doob-Meyer decomposition (\ref{eqn:Doob}) to rewrite under the expectation
	$$
		\mathbb{E}\Big[\int_0^TP_t^*\cdot  d\Sigma^N_t\Big] = \mathbb{E}\Big[\int_0^TP_t^*\cdot  b(\Sigma^N_t,\alpha^N_t)dt\Big].
	$$
	We obtain
\begin{align*}
	(I_1)=&\   \mathbb{E}\Big[\int_0^T\Big(- P_t^*\cdot \Big(b(\Sigma_t^*,\alpha_t^*) - b(S_t^*,A_t^*) - Db(S_t^*,A_t^*)\big[(\Sigma^N_t-S^*_t,\alpha^N_t-A^*_t)\big]\Big)dt\Big]\\
	=&\ \mathbb{E}\Big[\int_0^T\Big(- \frac{1}{2}P_t^*\cdot D^2b(\Sigma_t^*,\alpha_t^*)\big[(\Sigma^N_t-S^*_t,\alpha^N_t-A^*_t),(\Sigma^N_t-S^*_t,\alpha^N_t-A^*_t)\big]\Big)dt\Big]\\
	&\ +   o\big(\sup_{t\in [0,T]}\mathbb{E}\big[|\Sigma^N_t-S^*_t|^2\big]+\mathbb{E}\big[\int_0^T|\alpha^N_t-A^*_t|^2dt\big]\big),
\end{align*}
and conclude the proof.
\end{proof}

We next show that solutions to the LQG problem can be approximated with $N$-agents.  This result also shows how the information is incorporated into an asymptotically optimal approximate control.

\begin{proposition}	\label{prop:approximate_solution}
	We assume \ref{itm:continuity} - \ref{itm:smooth_data},  $(S^*,A^*)$ is the global minimum of the mean-field problem, $R_t>0$ and $\tilde{\Theta}_t>0$, and a solution to {\normalfont (\ref{eqn:Ricatti})} exists on $[0,T]$. We suppose that $N^{1/2}(\Sigma_0^N-S_0)\rightarrow^d \mathfrak{s}_0$ as $N\rightarrow \infty$, and $\lim_{N\rightarrow \infty}N\, \mathbb{E}\big[\Sigma_0^N-S_0^*\big]=\zeta_0$.

	We define the approximate Kalman filter and control by
	\begin{align}\label{eqn:approximate_Kalman}
		d\hat{\mathfrak{s}}^N_t = E_t\, \hat{\mathfrak{s}}_t^N\, dt + B_t\, N^{1/2}(\hat{\alpha}^N_t-A^*_t)\, dt +  \Pi_t\, \tilde{E}_t^\top \tilde{\Theta}^{-1}_t\, \big(N^{1/2}(d\Upsilon_t^N - \tilde{b}(S^*_t)dt) - \tilde{E}_t\, \hat{\mathfrak{s}}^N dt\big),
	\end{align}
	and
	\begin{align}\label{eqn:approximate_control}
		\hat{\alpha}^N_t =  A_t^*-N^{-1/2}\, R^{-1}_t\,  (Z_t\, B_t + V_t)^\top \hat{\mathfrak{s}}_t^N.
	\end{align}

	Let $(\mathfrak{s}^*,\mathfrak{a}^*,\mathfrak{u}^*)$ be the optimal solution of the LQG problem defined above.	Then, with $\hat{\Sigma}^N$ that solves {\normalfont(\ref{eqn:N_dynamics})}, the asymptotic cost for this approximate optimal control is given by 
	$$
		\lim_{N\rightarrow \infty}  N\, \Big(J^N\big[(\hat{\Sigma}^N),(\hat{\alpha}^N)\big] - J^{MF}\big[(S^*),(A^*)\big]\Big)	= LQG\big[(\mathfrak{s}^*),(\mathfrak{a}^*)] - P_0^*\cdot \zeta_0,
	$$
	and $N^{1/2}(\Sigma^N - S^*) \rightarrow^d (\mathfrak{s}^*)$, $N^{1/2}(\Upsilon^N - U^*) \rightarrow^d (\mathfrak{u}^*)$, and $N^{1/2}(\hat{\alpha}^N - A^*) \rightarrow^d (\mathfrak{a}^*)$ as $N\rightarrow \infty$.
\end{proposition}
\begin{proof}

	There exist unique solutions locally for (\ref{eqn:Kalman_covariance}), and if we show this solution is bounded then it exists for all $[0,T]$.	We bound $\Pi_t$ from above by the solution of
	$$
		\frac{d}{dt}\tilde{\Pi}_t = E_t\, \tilde{\Pi}_t + \tilde{\Pi}_t\, E_t^\top + \Theta_t,
	$$ 
	and $\Pi$ is bounded from below by $0$, thus the unique solution to (\ref{eqn:Kalman_covariance}) exists.

	It\^{o}'s lemma implies that
	\begin{align*}
		\frac{d}{dt}\mathbb{E}\big[|\hat{\Sigma}^N_t-S_t^*|^2\big]=&\ \mathbb{E}\big[2(\hat{\Sigma}^N_t-S_t)\cdot \big(b(\hat{\Sigma}_t^N,\hat{\alpha}^N_t)-b(S^*_t,A^*_t)\big) + {\rm tr}\, \Theta(\Sigma_t^N,\hat{\alpha}^N_t)\big] \\
		\leq&\ C\, \Big(\mathbb{E}\Big[|\hat{\Sigma}^N_t-S_t^*|^2+|\hat{\alpha}^N_t-A_t^*|^2\Big]+1\Big).
	\end{align*}
 The terms $N\, \mathbb{E}\big[|\hat{\alpha}^N_t-A_t^*|^2\big]$ and $N\, \mathbb{E}\big[|\hat{\Upsilon}^N_t-U_t^*|^2\big]$ are similarly bounded. 	An application of Gronwall's inequality yields
	$$
		N\, \mathbb{E}\big[|\Sigma^N_t-S_t|^2\big] \leq \tilde{C}.
	$$

	These bounds, along with (\ref{eqn:Doob}) and bounds on the covariation (\ref{eqn:covariation}), imply tightness of the distributions for $N^{1/2}(\Sigma^N_t-S_t)$ as well as $N^{1/2}(\hat{\alpha}^N_t-A^*_t)$ and $N^{1/2}(\Upsilon^N_t - U_t^*)$, using the tightness criteria of \cite{meyer1984tightness}.

	We compute that the quadratic covariation of $N^{1/2}(\Sigma^N_t-S_t)$ approaches the infinitesimal covariance $\Theta(S^*_t,A^*_t)$.  Theorem 1.3 of \cite{dawson1991law} implies that the  distributions converge $N^{1/2}(\Sigma^N - S^*) \rightarrow^d (\mathfrak{s}^*)$, $N^{1/2}(\Upsilon^N - U^*) \rightarrow^d (\mathfrak{u}^*)$, and $N^{1/2}(\hat{\alpha}^N - A^*) \rightarrow^d (\mathfrak{a}^*)$ as $N\rightarrow \infty$.  We then have the convergence of the cost by Lemma \ref{lem:cost_asymptotics}.  
\end{proof}

We would like to continue with a final result that gives a lower bound to the first-order asymptotics of the finite $N$-agent problem by the LGQ problem, which would complete our characterization of the fluctuations. We will sketch an argument for such a result, but cannot provide a complete proof due to a lack of full understanding of a key technical component.

\begin{theorem_star}\label{thm:CLT_partial_info}
	We assume \ref{itm:continuity} - \ref{itm:smooth_data}, $(S^*,A^*)$ is the global minimum of the mean-field problem, $R_t>0$ and $\tilde{\Theta}_t>0$, and {\normalfont(\ref{eqn:Ricatti})} has a solution on $[0,T]$. We suppose that $N^{1/2}(\Sigma_0^N-S_0^*)\rightarrow^d \mathfrak{s}_0$, which is normally distributed with zero mean and finite covariance $\Pi_0$ and $N\, \mathbb{E}\big[\Sigma_0^N-S_0^*\big]\rightarrow \zeta_0$. We then let $\mathfrak{s}^*$ and $\mathfrak{a}^*$ denote optimal solutions of the linear-quadratic-Gaussian approximation.

	Then the first-order asymptotic formula for the cost holds that
	$$
		\lim_{N\rightarrow \infty}  N\, \Big(J\big[(\Sigma^N),(\alpha^N)\big] - J^{MF}\big[(S^*),(A^*)\big]\Big)	= LQG\big[(\mathfrak{s}^*),(\mathfrak{a}^*)] - P_0^*\cdot\zeta_0,
	$$
	and $N^{1/2}(\Sigma^N-S^*)\rightarrow^d \mathfrak{s}$.
\end{theorem_star}

Given minimizers of the finite $N$-agent problem, $(\Sigma^N,\alpha^N,U^N)$ we consider $(\mathfrak{s}^N) = N^{1/2}\big((\Sigma^N)-(S^*)\big)$, $(\mathfrak{u}^N) =  N^{1/2}\big((\Gamma^N)-(U^*)\big)$, and $(\mathfrak{a}^N) = N^{1/2}\big((\alpha^N)-(A^*)\big)$. Based on the principles of tightness from \cite{aldous1989stopping} and estimates using asymptotic coercivity of LQG with respect to
$$
			\sup_{t\in [0,T]} \mathbb{E}\Big[|\mathfrak{s}^N_t|^2\Big] + \mathbb{E}\Big[\int_0^T |\mathfrak{a}_t^N|^2\, dt\Big] \leq C,
$$
one can obtain a weak limit of $(\mathfrak{s}^N,\mathfrak{u}^N) \rightarrow^d (\mathfrak{s}^N,\mathfrak{u}^N)$. 

We define the measure-valued process $\eta_t^N\in \mathcal{P}_2(\R^l)$ as the conditional distribution
		$$
			\int_{\R^l} \phi(x)\eta_t^N(dx) = \mathbb{E}\big[\phi(\mathfrak{s}_t^N)|\mathcal{G}^N_t\big].
		$$
The approach for Markov control policies has been developed in \cite{fleming1984stochastic} and implemented for the `closed-loop' mean-field convergence problem in \cite{lacker2018convergence}, and we show how this can be adapted to Markov control policies of the conditional distributions $(\eta^N)$.
		Proceeding as in Theorem \ref{thm:mean_field}, we compactify the control variable as a function of the conditional distribution as $\mu_t^N \in \mathcal{P}(\mathcal{P}_2(\R^l) \times \R^m)$, such that for $f\in C_b(\mathcal{P}_2(\R^l) \times \R^m)$, with now the conditional-Markovian interpretation that
\begin{align}\label{eqn:control_compactification}
	\int_{\mathcal{P}_2(\R^l) \times \R^m}f(\eta,a)\mu_t^N(d\eta, da) = \bar{\mathbb{E}}\big[f(\eta_t^N,\alpha_t^N)\big].
\end{align}
This allows us to obtain a weak limit $\mu_t$ from which we can disintegrate into maps $\hat{\mu}_t:\mathcal{P}_2(\R^l)\rightarrow \mathcal{P}(\R^m)$ and finally consider the control policy
$$
\mathfrak{a}_t = \int_{\R^m} a\, \hat{\mu}_t(\eta_t)(da).
$$

The cost has the form, using the law of total expectation and that $\mathfrak{a}_t^N$ is $\mathcal{G}_t^N$-adapted,
\begin{align*}
LQG\big[(\mathfrak{s}^N),(\mathfrak{a}^N)\big]=   \int_0^T\mathbb{E}\big[f_t(\eta_t^N,\mathfrak{a}_t^N)\big]dt + \mathbb{E}\big[\mathfrak{s}_T\cdot F\, \mathfrak{s}_T\big]
\end{align*}
where
$$
f_t(\eta,\mathfrak{a}) =\int_{\R^l}\big(s\cdot Q_t\, s + 2\, s\cdot V_t\, \mathfrak{a} + \mathfrak{a}\cdot R_t\, \mathfrak{a}\big)\eta(ds).
$$
Since $\mathfrak{a}\mapsto f_t(\eta,\mathfrak{a})$ is convex (recall we have assumed that $R_t>0$) it follows that 
$$
\lim_{N\rightarrow \infty} LQG\big[(\mathfrak{s}^N),(\mathfrak{a}^N)\big]\geq LQG\big[(\mathfrak{s}),(\mathfrak{a})\big].
$$
We also have that (\ref{eqn:linearized_dynamics}) holds.

Our key technical obstacle is that the information constraint, $\mathfrak{a}_t=\hat{\mathfrak{a}}_t\big((\mathfrak{u}_s)_{s\in[0,t]}\big)$, requires convergence of $\eta^N_t$ to $\eta_t$ that is the conditional probability distribution of $(\mathfrak{s})$ given $(\mathcal{G})$, the filtration generated by $(\mathfrak{u})$. It is well-known that the weak convergence of $(\mathfrak{s})$ and $(\mathfrak{u})$ is not sufficient to guarantee that $(\eta)$ is $(\mathcal{G})$-adapted.

We can mention a couple of works that address this issue but have been able to fully understand a resolution for our problem. The ambitious work of \cite{aldous1981extended} develops a theory of extended weak convergence, under which the conditional distributions converge. It is found that the continuity of paths of the limit process plays a fundamental role. Further development related to the convergence of conditional expectations was undertaken in \cite{coquet2001weak} and \cite{memin2003stability}. Extended weak convergence has had success in the analysis of optimal stopping as well as backward stochastic differential equations. The recent work \cite{backhoff2020all} has shown that extended weak convergence coincides with other definitions of an `adapted' weak topology when the time is discrete.  In particular, the adapted Wasserstein distance minimizes the expected distance between the `causal couplings' of probability measures on the probability space and provides a natural metric for the extended weak convergence topology.

\section{Examples}\label{sec:examples}

\subsection{Ising Game}\label{sec:examples_Ising}

	In the Ising example, there are two states for each agent and two global controls that determine the rate of transitions between the states. It is inspired by the physical Ising model and provides an example of a critical phase transition.  We consider the following parameters of the model:
	\begin{itemize}
		\item $\bs{\beta}$ governs the cost of deviating the control from the rest state of $1$.  (Small $\bs\beta$ corresponds to high cost).
		\item $\bs{H}$ adds an external bias to a preferred state.
		\item $\bs{J}$ adds a preference to congregate in one state (when $\bs{J}>0$).
		\item $\bs{q}$ governs the rate of observations of the states.
	\end{itemize}

	  The controls are exactly the transition rates, i.e., $\beta(0,1,\Sigma,\alpha) = \alpha^0$ and $\beta(1,0,\Sigma,\alpha) = \alpha^1$.  We assume the cost is
	\begin{align*}
		L(\Sigma,\alpha) =&\ \bs{\beta}^{-1} \Sigma^0\, \alpha^0\, \big(\log(\alpha^0)-1\big)+\bs{\beta}^{-1} \Sigma^1\, \alpha^1\, \big(\log(\alpha^1)-1\big)\\
		&\ - \mb{H}\,  (\Sigma^1-\Sigma^0) -\frac{1}{2}\mb{J}\, (\Sigma^1-\Sigma^0)^2.
	\end{align*}

	We reduce the problem to a single dimension by $\hat{\Sigma} = \Sigma^1 - \Sigma^0$ ($\Sigma^0 = \frac{1-\hat{\Sigma}}{2}$ and $\Sigma^1 = \frac{1+\hat{\Sigma}}{2}$).  We now have that $\hat{\Sigma}$ jumps by $\frac{2}{N}$ at rate $N\, \alpha^0\, (1-\hat{\Sigma})/2$ and jumps by $-\frac{2}{N}$ at rate $N\, \alpha^1\, (1+\hat{\Sigma})/2$.

	We will observe measurements of each particle, $\upsilon\in \{0,1\}$ with rate $q$ so that
	$$
		{\beta}(0,0,\Sigma) = \mb{q} \hbox{ and } {\beta}(1,1,\Sigma) = \mb{q}.
	$$
	for some constant $\mb{q}\geq 0$.

	We work in the reduced form with $S_t$ the mean-field limit of $\hat{\Sigma}$.  The dynamics of the mean-field system simplify to
	$$
		\frac{dS_t}{dt} = A^0\, (1-S_t) -A^1\, (1+S_t),
	$$
	with cost
	$$
		L(S,A) = \bs{\beta}^{-1}\frac{1-S}{2} A^0\big(\log(A^0)-1\big)+\bs{\beta}^{-1}\frac{1+S}{2} A^1\big(\log(A^1)-1\big) - \mb{H}\, S - \frac{\mb{J}}{2}\, S^2.
	$$

	If we consider the difference of the number of measurements, it evolves by
	$$
		\frac{dU}{dt} = \mb{q}\, S_t.
	$$

	The optimal control given co-state $P$ is
	$$
		A^0 = \exp\big\{2\, \bs\beta\, P\big\}.
	$$
	Similarly,
	$$
		A^1 = \exp\big\{-2\, \bs\beta\, P\}.
	$$
	The Hamiltonian reduces to, 
	\begin{align*}
		 H(S,P) = 
		 \bs{\beta}^{-1}\big(\cosh(2\, \bs\beta\, P)-S\, \sinh(2\, \bs\beta\, P)\big)+\mb{H}\, S+\frac{1}{2}\mb{J}\, S^2.
	\end{align*}

	The Hamiltonian flow for the mean-field limit is
	\begin{align*}
		\frac{dS_t}{dt} =&\ 2\big(\sinh(2\, \bs\beta\, P_t)-S_t\, \cosh(2\, \bs\beta\, P_t)\big),\\
		-\frac{dP_t}{dt} =&\ -\bs{\beta}^{-1}\sinh(2\, \bs\beta\, P_t)+\mb{H}+\mb{J}\, S_t.
	\end{align*}
	When $\mb{H}=0$, the critical points are solved simply by $S=0$ and $P=0$, and, when $\bs{\beta}\, \bs{J} \geq 1$,
	$$
		\sinh(2\, \bs\beta\, P) = \bs{\beta}\, \mb{J}\, S
	$$
	so
	$$
		\cosh(2\, \bs\beta\, P) = \frac{\sinh(2\, \bs\beta\, P)}{\tanh(2\, \bs\beta\, P)} = \bs{\beta}\, \mb{J},
	$$
	thus
	$$
		\bs{\beta}\, \mb{J} = \sqrt{1 + \bs{\beta}^2\mb{J}^2\, S^2}
	$$
	so
	$$
		S = \pm \sqrt{1 -\bs{\beta}^{-2}\mb{J}^{-2}}.
	$$
	We have
	$$
		P = 2^{-1}\bs\beta^{-1}\sinh^{-1}\big(\bs{\beta}\, \mb{J}\, S\big).
	$$

We will consider the equilibrium at $S=0$.  First we compute
$$
	\Theta(S,A) =  A^0\, (1-S)+	A^1\, (1+S)=2,
$$
and
$$
	\tilde{\Theta} =  \mb{q}\, \frac{(1-S)}{2}+	\mb{q}\, \frac{(1+S)}{2}=\mb{q}.
$$
We then calculate

\begin{align*}
	R =&\ \frac{1}{2} D_{AA}L(S^*,A^*)-\frac{1}{2}P^*D_{AA}b(S^*,A^*) \\
	=&\  \frac{1}{4\, \bs\beta}\left(\begin{array}{cc} \frac{1-S}{A^0} & 0 \\ 0 & \frac{1+S}{A^1}\end{array}\right)\\
	=&\ 4^{-1}\, \bs\beta^{-1}\left(\begin{array}{cc} 1 & 0 \\ 0 & 1\end{array}\right).\\
	B =&\ D_Ab(S,A)=\left(\begin{array}{c} 1 -S\\ -1-S\end{array}\right)=\left(\begin{array}{c} 1 \\ -1\end{array}\right).\\
	E =&\ D_Sb(S,A)=-A^0-A^1=-2.\\
	Q =&\ \frac{1}{2}D_{SS} L(S^*, A^*)- \frac{1}{2}P^*\, D_{SS}b(S^*,A^*)\\
	=&\ -\frac{\mb{J}}{2}.\\
	V =&\ D_{SA} L(S^*,A^*) - D_{SA}b(S^*,A^*)P^*\\
	=&\ \left(\begin{array}{c} 0 \\ 0\end{array}\right).\\
	\tilde{E} =&\ D_S\tilde{b}(S^*) = \mb{q}.
\end{align*}
The error covariance matrix, $\Pi$,  in equilibrium solves
\begin{align*}
	0 = & 2\, E\, \Pi - \tilde{E}^2\tilde{\Theta}^{-1}\Pi^2+\Theta\\
	0 =&\ -\mb{q}\, \Pi^2 - 4 \Pi + 2.
\end{align*}

We want the positive root which simplifies to
$$
	\Pi = \frac{1}{1+\sqrt{1+\frac{\mb{q}}{2}}}.
$$
Of course, the covariance becomes smaller when $\mb{q}$ is larger and more measurements are available.  When $\mb{q}=0$, $\hat{\mathfrak{s}}=0$, there is no benefit in deviating from the mean-field control and (although this no longer satisfies the assumptions of Theorem \ref{thm:CLT_partial_info}) $\Pi$ is simply the covariance of $\mathfrak{s}$ given the dynamics
$$
	d\mathfrak{s}_t = -2\, \mathfrak{s}_t\, dt + 2\, dW_t.
$$

To solve for $Z$ we have
\begin{align*}
0 =&\ 2E\, Z - Z\, B\, R^{-1}\, B^\top\, Z + Q\\
0=&\ -4\, Z -8\, \bs\beta\, Z^2 - \frac{\mb{J}}{2}.
\end{align*}
We solve to get
\begin{align*}
	Z = \frac{4\pm \sqrt{16 -16\, \bs\beta\, \mb{J}}}{-16\, \bs\beta},
\end{align*}
so the solution exists if $\bs\beta\, \mb{J}\leq 1$ and is given by
$$
	 Z=4^{-1}\, \bs\beta^{-1}\, (-1\pm\sqrt{1-\bs\beta\, J}).
$$

In the critical case, $\bs\beta\, \mb{J}=1$, $Z=4^{-1}\, \bs\beta^{-1}$. 
The evolution is given by
$$
	d\mathfrak{s}_t = -2\, \mathfrak{s}_t -BR^{-1}\, Z(B+V)\hat{\mathfrak{s}}_t + 2\, dW_t
$$
and $BR^{-1}B\, Z = -2$ so, since $\hat{\mathfrak{s}}_t$ remains near $\mathfrak{s}$, the drift cancels over long times while the fluctuations grow.

\subsection{SIR Epidemic Model}

	As a second example we consider a compartmental epidemic model.  There are three states: 
	\begin{itemize}
		\item $\Sigma^0$, susceptible,
		\item $\Sigma^1$, infectious, 
		\item $\Sigma^2$, recovered,
	\end{itemize}
	and one control $\alpha$ (a social distancing parameter).  The additional parameters of the problem are
	\begin{itemize}
		\item $\bs\gamma$ is the recovery rate.
		\item $\mb{b}$ is the base line infection rate.
		\item $\mb{k}$ is a coefficient of the cost to reduce $\alpha$ below $\mb{b}$.
		\item $\mb{c}$ is a cost of infections.
		\item $\bs\nu$ rate of testing of infected individuals.
	\end{itemize}

	We then suppose that a transition from susceptible to infectious occurs as $\beta(0,1,\Sigma,\alpha) = \alpha\, \Sigma^1$, and a transition from infectious to recovered occurs at rate $\beta(1,2,\Sigma,\alpha) =  \bs\gamma$.  

	We assume the infected individuals are tested at a rate $\bs\nu$,
	$$
		\tilde{\beta}(1,0,\Sigma) = \bs\nu.
	$$

	The cost is
	\begin{align*}
		L(\Sigma,\alpha) =&\ \mb{k}\, \big(-\log(\frac{\alpha}{\mb{b}})+\frac{\alpha}{\mb{b}}-1\big) +\mb{c}\,\Sigma^1.
	\end{align*}

	A similar cost was used in \cite{palmer2020optimal} with applications to the COVID-19 pandemic.  We can consider the problem with a terminal condition, but for simplicity we set $G(\Sigma) = 0$.

For the mean field SIR example, the mean-field dynamics are
\begin{align*}
\frac{dS_t^0}{dt} =&\  - A_t\, S_t^0\, S_t^1\\
\frac{dS_t^1}{dt} =&\ A_t\, S_t^0\, S_t^1 - \bs\gamma\, S_t^1\\
\frac{dS_t^2}{dt} =&\ \bs\gamma\, S_t^1.
\end{align*}

The number of confirmed tests evolves simply by
$$
	\frac{dU_t}{dt} = \bs\nu\, S_t^1.
$$

The recovered state, $(S^2)$, is irrelevant so we will ignore it. The Hamiltonian is then
	\begin{align*}
		\mathcal{H}(S,A,P) =&\ A\, S^0\, S^1\, (P^1 - P^0) - \mb{k}\big(-\log(\frac{A}{\mb{b}})+\frac{A}{\mb{b}} -1\big)-\mb{c}\, S^1.
	\end{align*}
	The control satisfies
	$$
		A = \frac{\mb{b}\, \mb{k}}{\mb{k} - \mb{b}\, S^0\, S^1\, (P^1-P^0)}.
	$$

	Plugging this back in we get
	\begin{align*}
		H(S,P) 
		= -\mb{k}\Big(\log\big(1-\frac{\mb{b}\, S^0\, S^1}{\mb{k}}(P^1-P^0)\Big)  - \mb{c}\, S^1.
	\end{align*}

	The Hamiltonian equations are
	\begin{align*}
		\frac{dS^0_t}{dt} =&\ \frac{ - \mb{b}\, \mb{k}\, S^0_t\, S^1_t}{\mb{k} - \mb{b}\, S^0_t\, S^1_t\, (P^1_t-P^0_t)}\\
		\frac{dS^1}{dt}=&\ \frac{ \mb{b}\, \mb{k}\,S^0_t\, S^1_t}{\mb{k} - \mb{b}\, S^0_t\, S^1_t\, (P^1_t-P^0_t)} -\bs\gamma\, S_t^1,
	\end{align*}
	and
	\begin{align*}
		-\frac{dP^0_t}{dt} =&\   \frac{\mb{k}\, \mb{b}\, S^1_t\, (P^1_t-P^0_t)}{\mb{k} - \mb{b}\, S^0_t\, S^1_t\, (P^1_t-P^0_t)}  \\
		-\frac{dP^1_t}{dt}=&\  \frac{\mb{k}\, \mb{b}\, S_0^t\, (P^1_t-P^0_t)}{\mb{k} - \mb{b}\, S^0_t\, S^1_t\, (P^1_t-P^0_t)}- \bs\gamma\, P^1_t - \mb{c}.
	\end{align*}
	Equilibria occur when $S^1=0$ and
	\begin{align*}
		0 =&\  \mb{b}\, S^0 (P^1-P^0) - \bs\gamma\, P^1 -\mb{c}.
	\end{align*}

	Due to the possible instability at the equilibrium, linearizing can lead to very bad results.  

	We show numerical results based on solving the Hamiltonian equations and Kalman filter numerically in discrete time ($\Delta t=1$; see Section \ref{sec:discrete_time}).  We select as parameters $\mb{b}=0.87$, $\bs\gamma=0.217$, $\bs\nu=\frac{1}{3}$, $\mb{c}=8,000$, $\mb{k}=100$, $N=10,000$, and $T=100$. Results are shown in Figure \ref{fig:sir}.

	\begin{figure}
	\centering
	$
		\includegraphics[scale=0.7]{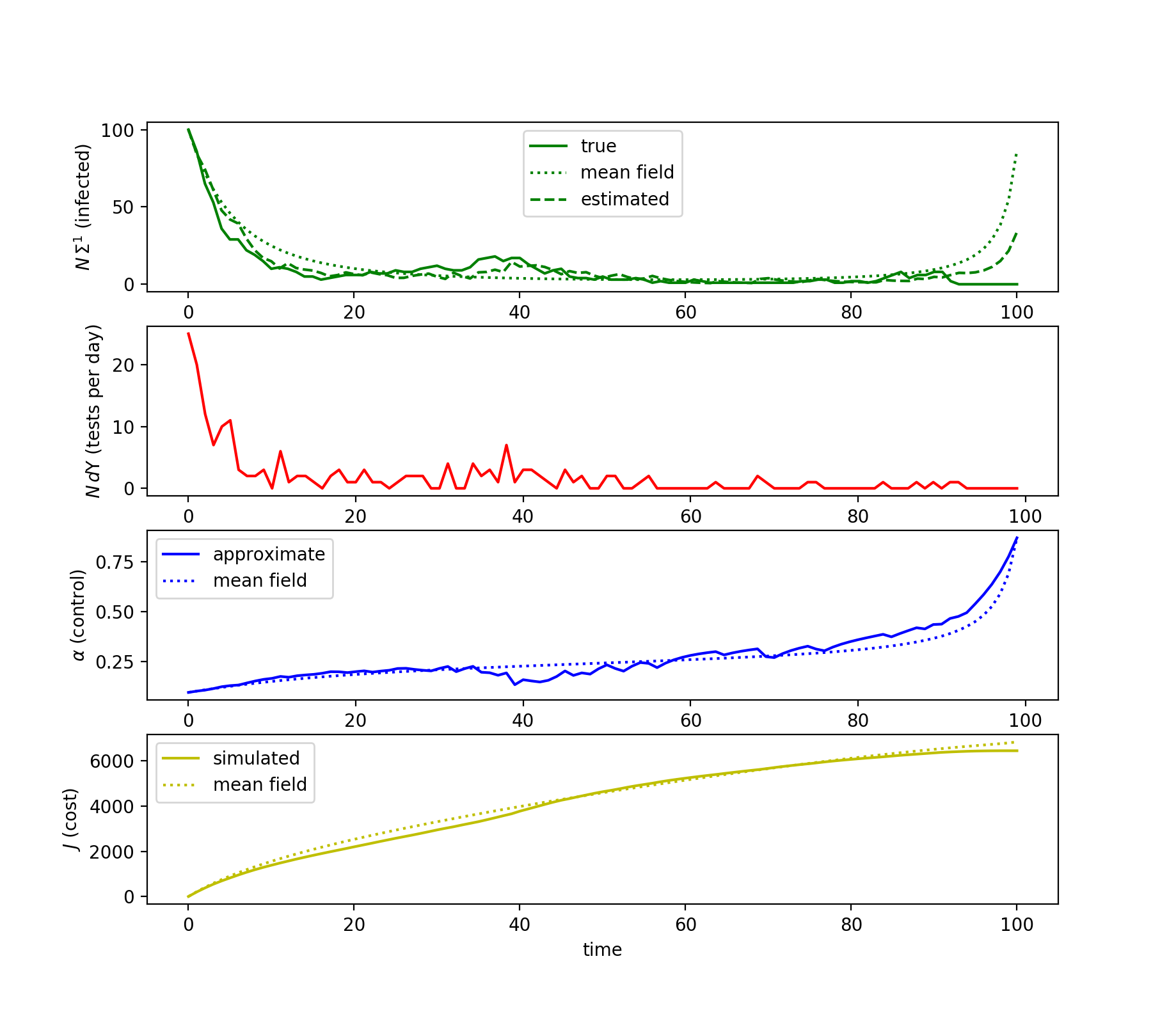}
	$
	\caption{\label{fig:sir} The infected population, $\Sigma^1$, tests per day, and the approximate control from Proposition \ref{prop:approximate_solution} for the SIR model.}
	\end{figure}

\section{Discrete Time}\label{sec:discrete_time}

All of our analysis also applies to a problem in discrete time.  We present it in a way that is a discretization of our continuous time problem, although it is not necessary that the time steps are small.  We suppose that at time $t = k\, \Delta t$ the transitions at time $t+\Delta t$ from state $\sigma$ to $\gamma$ occur with probability $\beta(\sigma,\gamma, \alpha_k,\Sigma_k)\, \Delta t$.  We require that $\sum_{\gamma} \beta(\sigma, \gamma,\alpha_k,\Sigma_k)\, \Delta t\, < 1$.  We assume the cost has the form
$$
	J^N\big[(\Sigma),(\alpha)\big] = \mathbb{E}\Big[\sum_{k=0}^{T-1} L(\Sigma_k,\alpha_k)\Delta t + G(\Sigma_{T})\Big].
$$

When $N$ is large and $\Delta t$ is small, the number of agents transitioning from $\sigma$ to $\gamma$ is well approximated by a Poisson distribution of rate $\beta(\sigma,\gamma,\alpha_k,\Sigma_k)\, \Sigma_k^\sigma\, \Delta t$.  We then can use the same definitions for $b(\Sigma,\alpha)$, $\tilde{b}(\Sigma)$, $\epsilon(\Sigma,\alpha)$, and $\tilde{\epsilon}(\Sigma)$.

The mean-field problem corresponds to the discretized dynamics
$$
	S_{k+1} = S_k + b(S_k,\alpha_k)\, \Delta t. 
$$

The co-state, $P_{T} = -D_SG(\Sigma_{T})$,
$$
	P_k = P_{k+1} + D_SH(S_k,P_{k+1})\Delta t
$$
allows us to compute the gradient of the cost as
\begin{align*}
	D_{A_k}J^{MF}[(S),(A)] = -\Big(P_{k+1} \cdot D_Ab(S_k,A_k) - D_AL(S_k,A_k\Big)\Delta t.
\end{align*}

The statement and proof of Theorem \ref{thm:mean_field} is now essentially the same.

To express the linear quadratic problem to describe the fluctuations, we again define the following
\begin{align}\label{eqn:partial_info_relations_discrete}
	R_k =&\  \Big(\frac{1}{2}D_{AA}L(S^*_k,A^*_k)-\frac{1}{2}P_{k+1}^*\cdot D_{AA}b(S^*_k,A^*_k)\Big)\Delta t \\
	B_k =&\ D_A b(S^*_k,A^*_k)\, \Delta t\nn\\
	E_k =&\ D_S b(S^*_k,A^*_k)\, \Delta t\nn\\
	Q_k =&\ \Big(\frac{1}{2}D_{SS} L(S^*_k, A^*_k)- \frac{1}{2}P^*_{k+1}\cdot D_{SS}b(S^*_k,A^*_k)\Big)\, \Delta t\nn\\
	W_k =&\ \Big(\frac{1}{2}\Big(D_{SA} L(S^*_k,A^*_k) - P^*_{k+1}\cdot D_{SA}b(S^*_k,A^*_k)\Big)\Delta t\nn\\
	\tilde{E}_k =&\ D_S \tilde{b}(S^*_k)\, \Delta t\nn\\
	\Theta_k =&\ \epsilon(S^*_k,A^*_k)\epsilon(S^*_k,A_k)^\top\, \Delta t\nn\\
	\tilde{\Theta}_k =&\ \tilde{\epsilon}(S^*_k)\epsilon(S^*_k)^\top\, \Delta t\nn\\
	F =&\ \frac{1}{2}D^2G(S_T^*)\nn
\end{align}
The discrete form of the covariance equation, (\ref{eqn:Kalman_covariance}), takes the form 
\begin{align}\label{eqn:Kalman_covariance_discrete}
	\Pi_{k+1} = (I+E_k)\Big(\Pi_k - \Pi_k\, \tilde{E}_k^\top(\tilde{E}_k \, \Pi_k\, \tilde{E}_k^\top+ \tilde{\Theta}_k)^{-1} \tilde{E}_k\, \Pi_k\Big)(I+E_k)^\top+ \Theta_k,
\end{align}
and the estimator is given by
\begin{align}\label{eqn:Kalman_estimator_discrete}
	\hat{\mathfrak{s}}_{k+1} =&\ \hat{\mathfrak{s}}_k + E_k\, \hat{\mathfrak{s}}_k + B_k\, \mathfrak{a}_k \\
	&\ + \Pi_{k+1}\, \tilde{E}_{k+1}^\top(\tilde{E}_{k+1} \, \Pi_{k+1}\, \tilde{E}_{k+1}^\top+ \tilde{\Theta}_{k+1})^{-1}\big(\mathfrak{u}_{k+1}-\tilde{E}(\hat{\mathfrak{s}}_k+E_k\, \hat{\mathfrak{s}}_k + B_k\, \mathfrak{a}_k)\big).\nn
\end{align}

The backward Ricatti equation is
\begin{align}\label{eqn:backward_Ricatti_discrete}
	Z_{k} = (I+E_k)^\top\Big(Z_{k+1} - Z_{k+1}\, B_k^\top(B_k \, Z_{k+1}\, B_k^\top+ R_k)^{-1} B^\top_k\, Z_{k+1}\Big)(I+E_k)^\top+ Q_k,
\end{align}
and the optimal control is
$$
	\mathfrak{a}_k = -(B_k \, Z_{k+1}\, B_k^\top+ R_k)^{-1} B^\top_k\, Z_{k+1}\hat{\mathfrak{s}}_k.
$$

For numerical implementation, we use a gradient descent algorithm to find the optimal mean-field solution, and proceed directly to solve (\ref{eqn:Kalman_estimator_discrete}) and (\ref{eqn:backward_Ricatti_discrete}).

\bibliography{OptCLTBIB}
\bibliographystyle{plain}

\end{document}